\documentclass{amsart}

\usepackage[mathscr]{eucal}
\usepackage{amssymb}
\usepackage{graphicx}
\usepackage[usenames,dvipsnames]{color}
\usepackage[normalem]{ulem}
\usepackage{amsthm}
\usepackage{bbold}
\usepackage{enumerate}
\usepackage{array}
\usepackage{amsmath}
\usepackage{mathtools}

\usepackage{hyperref}

\hypersetup{colorlinks=true,citecolor=BrickRed,urlcolor=BrickRed,linkcolor=BrickRed}

\usepackage[all]{xy}
\CompileMatrices

\newdir{ >}{{}*!/-10pt/\dir{>}}

\definecolor{ForestGreen}{rgb}{0.13,.57,.13}

\newcommand{\id}{\mathrm{id}}

\newcommand{\soc}{\mathop{\mathrm{soc}}}

\newcommand{\End}{\mathrm{End}}

\newcommand{\inv}{^{-1}}

\newcommand{\Hom}{\mathrm{Hom}}
\newcommand{\Ext}{\mathrm{Ext}}

\newcommand{\Z}{\mathbb{Z}}

\newcommand{\F}{\mathbb F}

\newcommand{\la}{\longrightarrow}

\theoremstyle{definition}

\newtheorem*{conv*}{Conventions}
\newtheorem*{ack}{Acknowledgements}

\newtheorem*{hyp*}{Hypothesis}
\newtheorem{example}[equation]{Example}

\newtheorem*{conventions*}{Conventions}

\theoremstyle{plain}

\newtheorem{thm}[equation]{Theorem}
\newtheorem{lemma}[equation]{Lemma}
\newtheorem{thm-defi}[equation]{Theorem-Definition}

\newtheorem*{thm*}{Theorem}
\newtheorem*{lemma*}{Lemma}
\newtheorem*{cor*}{Corollary}

\theoremstyle{remark}

\newtheorem*{remark*}{Remark}

\begin{document}


\title{Mackey algebras which are Gorenstein}
\author{Ivo Dell'Ambrogio}
\author{Jan \v S\v tov\'i\v cek}


\address{Ivo Dell'Ambrogio, Laboratoire de Math\'ematiques Paul Painlev\'e, Universit\'e de Lille~1, Cit\'e Scientifique -- B\^at.\,M2, 59665 Villeneuve-d'Ascq Cedex, France}
\email{ivo.dellambrogio@math.univ-lille1.fr}

\address{Jan \v S\v tov\'i\v cek, Department of Algebra, Charles University in Prague, Sokolovsk\'a~83, 186 75 Praha~8, Czech Republic}
\email{stovicek@karlin.mff.cuni.cz}

 \date{\today}

\begin{abstract}
We complete the picture available in the literature by showing that the integral Mackey algebra is Gorenstein if and only if the group order is square-free, in which case it must have Gorenstein dimension one. We illustrate this result by looking in details at the examples of the cyclic group of order four and the Klein four group.
\end{abstract}

\thanks{First-named author partially supported by the Labex CEMPI (ANR-11-LABX-0007-01)}
\thanks{Second-named author supported by grant GA\v{C}R P201/12/G028 from the Czech Science Foundation.}

\maketitle




Let $G$ be a finite group and $R$ a commutative ring of coefficients. 
The \emph{Mackey algebra} $\mu_R(G)$, introduced in~\cite{ThevenazWebb95}, is a finite-rank free $R$-algebra whose representations form precisely the category of Mackey functors for $G$ with $R$-coefficients. It can be defined as the endomorphism algebra of $\coprod_{H\leq G}G/H$ in the category of finite $G$-sets and $R$-linear combinations of spans of $G$-maps.
 In this note we will say that a (unital, associative, non-necessarily commutative) ring $S$ is \emph{Gorenstein}\footnote{If the base~$R$ is noetherian, then either the Mackey algebra $\mu_R(G)$ or the Burnside ring $RB(G)$ is Gorenstein (in our sense) iff it is \emph{Iwanaga-Gorenstein}, i.e.\ noetherian and of finite injective dimension on \emph{both} sides. Indeed, both are self-dual finite $R$-algebras; cf.\ \cite[Ex.\,2.2]{DSS17}.} if its injective dimension over itself as a left module is finite. The latter number $n$ is the \emph{Gorenstein dimension} of~$S$, and we will sometimes say for short that $S$ is \emph{$n$-Gorenstein}. Famously, Gorenstein rings are ubiquitous in mathematics: see \cite{Bass63} for the classical case of local commutative rings, and e.g.\ \cite{DSS17} for a wide-ranging look at non-commutative rings and small categories.

The goal of this note is to clarify the literature on the following point:

\begin{thm} \label{thm:characterisation_of_Gorenstein}
The integral Mackey algebra $\mu_\Z(G)$ is Gorenstein if and only if the order of $G$ is square-free, in which case it has Gorenstein dimension one.
\end{thm}

Note that this contradicts \cite[Lemma~2.2]{Greenlees92}, which states that projective Mackey functors over \emph{any} finite group have injective dimension one. The lemma in question was used to prove that the the finitistic dimension of $\mu_\Z(G)$ equals one, and a different proof of the latter theorem was provided (together with several other similar results) by Bouc, Stancu and Webb; see \cite[Theorem~1.2]{BSW17}. Still, the question whether $\mu_\Z(G)$ is always Gorenstein was left open. As it turns out, it is Gorenstein precisely for groups $G$ with square-free order. These are also the only groups for which $\mu_\Z(G)$ can be made into a symmetric $\Z$-algebra, by a theorem of Rognerud \cite{Rognerud15} (we will use this result, so more on this below); and they are also precisely the groups for which the Burnside ring $B(G)$ is (1-)Gorenstein or even a symmetric $\Z$-algebra, by theorems of Kr\"amer \cite{Kraemer74} and Gustafson \cite{Gustafson77}.

For one direction of our theorem we actually know somewhat more:

\begin{thm} \label{thm:Gorenstein_general_R}
The Mackey algebra $\mu_R(G)$ is Gorenstein of dimension~$n$, provided the order of $G$ is square-free and $R$ is Gorenstein of dimension~$n$.
\end{thm}

Thus for instance if $R$ is a Dedekind domain $\mu_R(G)$ must be 1-Gorenstein.

\begin{center}$***$\end{center}

Our proofs essentially consist of a standard argument for reducing modulo a regular element, Lemma~\ref{lemma:reduction} below, and of combining a few deeper known results.

\begin{proof}[Proof of Theorem~\ref{thm:Gorenstein_general_R}]
The main point of the proof was discovered by Gustafson \cite{Gustafson77}: when $|G|$ is square-free one can define a nice bilinear form $\beta_G$ on the integral Burnside ring $B(G):= K_0(G\textsf{-set}, \amalg, \times)$. More precisely: square-free implies solvable (\cite[Theorem 9.4.3]{Hall59} or \cite[Lemma 5.4]{Rognerud15}), which implies that there is precisely one conjugacy class of subgoups of $G$ of any given order $d$ dividing~$|G|$. Then the bilinear form $\beta_G\colon B(G)\times B(G)\to \Z$ is defined by sending $(x,y)$ to $[G/1]^*(x\cdot y)$, where $[G/1]^*\colon B(G)\to \Z$ is the form sending a basis element $[G/H]\in B(G)$ to $1$ if $H=1$ and to $0$ otherwise. The form is symmetric ($\beta_G(x,y)=\beta_G(y,x)$), non-degenerate (in that it induces and isomorphism $B(G)\stackrel{\sim}{\to}B(G)^*$ onto the $\Z$-dual) and associative ($\beta_G(x,y\cdot z)=\beta_G(x\cdot y,z)$).
By tensoring with~$R$, we obtain such a form over $RB(G)=R\otimes_\mathbb ZB(G)$ for any ring~$R$.

The second ingredient is due to Rognerud~\cite{Rognerud15}, who showed that to give an associative symmetric bilinear form $\beta$ over the Mackey algebra $\mu_R(G)$ amounts to giving a family $\{\beta_H\colon RB(H)\times RB(H)\to R\}_{H\leq G}$ of associative symmetric $R$-bilinear forms indexed by the subgroups of~$G$ and satisfying $\beta_G(1_{RB(G)}, \mathrm{ind}_H^G(x))=\beta_H(1_{RB(H)}, x)$. Since the Gustafson forms satisfy this compatibility, they yield a bilinear form on $\mu_R(G)$ which again is symmetric, non-degenerate and associative. (See \cite[\S3]{Rognerud15} for more details).

We may now immediately conclude with the third and final ingredient, namely the following criterion for the Gorenstein property.
\end{proof}

\begin{lemma} \label{lemma:symm_Gorenstein}
Let $R$ be an $n$-Gorenstein commutative ring, and let $A$ be an $R$-algebra which is finitely generated projective as an $R$-module. If $A$ admits a symmetric, non-degenerate and associative bilinear form $\beta\colon A\times A\to R$ and if the unit map $R\to A$ admits an $R$-linear retraction, then $A$ is also an $n$-Gorenstein ring.
\end{lemma}

\begin{proof}
This generalises the classical fact that a symmetric Frobenius algebra (a.k.a.\ symmetric algebra over a field) is self-injective, and is itself a special case of the much more general `relative Serre functor' criterion proved in \cite[Theorem~1.6]{DSS17}. Indeed, in the notation of \emph{loc.\ cit.}, we may take $\mathcal C$ to be~$A$, seen as an $R$-linear category with a single object whose endomorphism algebra is~$A$. Then the Serre functor $S$ is the identity morphism $A\to A$ and the isomorphism $\sigma$ of left $R$-modules
\[
\sigma\colon A \stackrel{\sim}{\longrightarrow} \Hom_R(A,R)
\]
is simply the one induced by~$\beta$, namely $x\mapsto (y\mapsto \beta(y , x))$. The naturality of $\sigma$ required in \emph{loc.\ cit.} means that it is an isomorphism of $A$-bimodules, and the latter amounts precisely to the associativity and symmetric properties of~$\beta$.
\end{proof}

\begin{proof}[Proof of Theorem \ref{thm:characterisation_of_Gorenstein}.]
Assume that $|G|$ is square-free. Then the integral Mackey algebra $\mu_\Z(G)$ is 1-Gorenstein, because $\Z$ has global dimension one and therefore in particular is 1-Gorenstein, so that we may apply Theorem~\ref{thm:Gorenstein_general_R} with $R=\Z$.

Conversely, assume that $\mu_{\Z}(G)$ is Gorenstein of Gorenstein dimension~$n$. Then for every prime number $p$ the mod~$p$ Mackey algebra $\mu_{\F_{\!p}\!}(G)$ is Gorenstein of dimension at most $n-1$. 
Indeed, $p$ is a central element and a non-zero divisor in $\mu_{\Z}(G)$ and $\mu_{\F_{\!p}\!}(G)=\mu_{\Z}(G)/p \cdot \mu_{\Z}(G)$. Thus if $N$ is any $\mu_{\F_{\!p}\!}(G)$-module we have 
\[
0 
=  \Ext^{i+1}_{\mu_{\Z}(G)}(N, \mu_{\Z}(G)) 
\cong \Ext^i_{\mu_{\F_{\!p}\!}(G)}(N,\mu_{\F_{\!p}\!}(G))
\]
for all $i\geq n$ by Lemma~\ref{lemma:reduction}. Then we conclude with \cite[Corollary~1.3]{BSW17}.

For the reader's convenience, let us recall the latter argument. By \cite[Theorem~1.2]{BSW17}, Mackey algebras over a field always have finitistic dimension zero. In particular, $\mu_{\F_{\!p}\!}(G)$ is Gorenstein if and only if it is self-injective, and self-injective Mackey algebras over a field have been characterised in~\cite[Theorem 19.2]{ThevenazWebb95}. Over $\F_{\!p}$, this can only happen when the $p$-Sylows of $G$ have order $1$ or~$p$, that is, when $p^2$ does not divide the order of~$G$. As $p$ was an arbitrary prime, we conclude that $|G|$ is square-free.
\end{proof}

\begin{lemma} \label{lemma:reduction}
Let $M$ and $N$ be two (left) $S$-modules over a (nonnecessarily commutative) ring~$S$. If $x$ is an $S$-regular, $M$-regular and central element of $S$ such that $xN=0$, then 
\[
\Ext^{i+1}_S(N,M) \cong \Ext^i_{S/xS}(N,M/xM)
\]
for all $i\geq 1$.
\end{lemma}

\begin{proof}
When $S$ is commutative, this is the well-known special case $k=1$ of \cite[Theorem 2.1]{Rees56}. However, as noted by H.\,Cartan in his MathSciNet review of Rees' paper, the hypothesis of commutativity is unnecessary and it suffices to assume $x$ central in~$S$. (See also \cite[{Lemma 3.1.16}]{BrunsHerzog93} for a more direct proof.)
\end{proof}

\begin{center}$***$\end{center}

We illustrate the above results explicitly with the two smallest groups which are not square-free: the cyclic group $C_4$ and the Klein group $C_2\times C_2$. Let us first inspect the corresponding modulo $2$ Burnside rings $\F_{\!2}B(C_4)$ and $\F_{\!2}B(C_2\times C_2)$.

A general theory in \cite{Gustafson77} says that given a finite $p$-group $G$, $\F_{\!p}B(G)$ is a local commutative finite dimensional algebra (\cite[Lemma 3]{Gustafson77}). Moreover, the socle of $\F_{\!p}B(G)$ contains the element $[G/1]$ and, if $|G| > p$, also the (different) non-zero element $\sum_{P\le G, |H|=p} [G/H]$ (\cite[pp.~11--14]{Gustafson77}).

\begin{example} \label{example:B_ring_C4}
The underlying group of the integral Burnside ring $\Z{}B(G)$ for $G=C_4$ has rank three and a basis $[G/G]$, $[G/H]$ and $[G/1]$, where $1<H<G$ and $H$ is cyclic of order two. The element $[G/G]$ is the unity of the ring. To ease the notation, we write $g=[G/1]$ and $h=[G/H]$. Then we obtain a ring isomorphism
\[ \Z{}B(C_4) \cong \Z[g,h]/(g^2-4g, h^2-2h, gh-2g) \]
by direct computation. The modulo $2$ version simplifies to
\[ \F_{\!2}B(C_4) \cong \F_{\!2}[g,h]/(g^2, h^2, gh)\,. \]
Observably $\soc \F_{\!2}B(C_4) = \F_{\!2}g \oplus \F_{\!2}h$ is two-dimensional, so this is not a self-injective ring (conclude e.g.\ with Lemma~\ref{lemma:soc_P} below, or use some local commutative algebra).
\end{example}

\begin{example} \label{example:B_ring_Klein}
The case $G = C_2\times C_2$ is slightly more complicated. The group $G$ has three non-trivial subgroups $H,K,L$ of order two. If we denote $g=[G/1]$, $h=[G/H]$, $k=[G/K]$ and $\ell=[G/L]$, the four non-unit basis elements of $\Z{}B(G)$, we can express the integral Burnside ring as the quotient of $\Z[g,h,k,\ell]$ by
\[
g^2=4g, \; gh=gk=g\ell=2g, \; h^2=2h, \; k^2=2k, \; \ell^2=2\ell, \; hk=h\ell=k\ell=g.
\]
Obviously, $g$ can be omitted from the set of generators and we obtain
\[ \Z{}B(C_2\times C_2) \cong \Z[h,k,\ell]/(h^2-2h, k^2-2k, \ell^2-2\ell, hk-h\ell, h\ell-k\ell). \]
The modulo 2 version takes the shape
\[ \F_{\!2}B(C_2\times C_2) \cong \F_{\!2}[h,k,\ell]/(h^2, k^2, \ell^2, hk-h\ell, h\ell-k\ell). \]
To see why it is not self-injective, it is convenient to replace the generator $\ell$ by $s=k+h+\ell$. After the substitution, we get
\[ \F_{\!2}B(C_2\times C_2) \cong \F_{\!2}[h,k,s]/(h^2, k^2, s^2, hs, ks). \]
This again is not a self-injective ring since $\soc \F_{\!2}B(C_2\times C_2) = \F_{\!2}s \oplus \F_{\!2}hk$.
\end{example}

\begin{center} $***$ \end{center}

Now we will explicitly see why the modulo $2$ Mackey algebras $\mu_{\F_{\!2}}(G)$ for $G=C_4$ and $G=C_2\times C_2$ are not self-injective, by means of the following observation:

\begin{lemma}
\label{lemma:soc_P}
The socle of any indecomposable projective left module $P$ over a finite dimensional self-injective algebra~$A$ is simple.
\end{lemma}
\begin{proof}
Since $A$ is self-injective, an indecomposable projective $P$ is also injective and therefore is the injective envelope of its socle, which must be simple because injective envelopes commute with direct sums.
\end{proof}

Here we compute using the description of $\mu_{\F_{\!2}}(G)$ as the endomorphism algebra of $\coprod_{H\leq G}G/H$ in the \emph{Burnside category} $\F_{\!2}\mathscr{B}(G)$ of finite $G$-sets and $\F_{\!2}$-linear combinations of isomorphism classes of spans of $G$-maps (see \cite[\S2]{ThevenazWebb95}; this goes back to~\cite{Lindner76}). In this picture, there is an isomorphism of algebras $\F_{\!2}B(G) \cong \End_{\F_{\!2}\mathscr{B}(G)}(G/G)$ sending $[G/H]$ to the isomorphism class of the span
\[
\xymatrix@R=10pt@C=10pt{
& G/H \ar[dl] \ar[dr] \\
G/G && G/G.
}
\]
In particular, we obtain a (non-unital!) inclusion of algebras $\F_{\!2}B(G) \hookrightarrow \mu_{\F_{\!2}}(G)$ which sends $x \in \End_{\F_{\!2}\mathscr{B}(G)}(G/G)$ to the endomorphism
\[ \coprod_{H\leq G}G/H \xrightarrow{\mathrm{proj}} G/G \overset{x}\la G/G \xrightarrow{\mathrm{inc}} \coprod_{H\leq G}G/H. \]
Let $P:=\mu_{\F_{\!2}}(G) \, e$ be the projective module defined by the idempotent $e=\id_{G/G}$. As already mentioned, its endomorphism ring $\End(P)\cong e\, \mu_{\F_{\!2}}(G) \, e \cong \F_{\!2}B(G)$ is local in our case, hence $P$ is indecomposable (\cite[I.4.7-9]{AuslanderReitenSmalo97}). Therefore, by Lemma~\ref{lemma:soc_P}, in order to show that the modulo $2$ Mackey algebras for $C_4$ and $C_2\times C_2$ are not selfinjective it is enough to prove that the two elements 
\begin{equation} \label{equation:vanishing-comp}
g := [G/1]  
\quad\quad 
\textrm{ and } 
\quad\quad 
s := \sum_{H\le G, |H|=p} [G/H]  \,,
\end{equation}
viewed inside of $e\,\mu_{\F_{\!2}}(G)\,e\subset \mu_{\F_{\!2}}(G)$ via the latter embedding, are in the socle of $\mu_{\F_{\!2}}(G)$, hence of~$P$.
To that end, it is enough to convince oneself that the composites
\begin{equation*} 
G/G \overset{g}\la G/G \overset{f}\la G/H
\qquad\textrm{and}\qquad
G/G \overset{s}\la G/G \overset{f}\la G/H
\end{equation*}
in $\F_{\!2}\mathscr{B}(G)$ vanish for any proper subgroup $H<G$ and morphism $f\colon G/G \to G/H$.

Moreover, by definition $f\colon G/G \to G/H$ is an $\F_{\!2}$-linear combination of spans
\[
\xymatrix@R=10pt@C=10pt{
& G/K \ar[dl] \ar[dr]^{f'} \\
G/G && G/H
}
\]
where $f'$ is a map of $G$-sets. Hence, the isomorphism class of this span is the same datum and the isomorphism type of $f'$ seen as a $G$-set over $G/H$. It is well known that each such $G$-set over $G/H$ is of the form $\mathrm{ind}_H^G({f'}\inv(\{1H\}) \to H/H)$, 
see for instance \cite[Lemma 2.4.1]{Bouc97}.
Thus, we have
\[ \Hom_{\F_{\!2}\mathscr{B}(G)}(G/G, G/H) \cong \F_{\!2}B(H) \]
as vector space. This observation is used below to compute a basis of the vector space $\Hom_{\F_{\!2}\mathscr{B}(G)}(G/G, G/H)$ in each of the two cases, $G=C_4$ and $G=C_2\times C_2$.

\begin{example} [Cont.\ Example~\ref{example:B_ring_C4}]
\label{example:Mackey_C4}
For $G=C_4> H>1$, $\Hom_{\F_{\!2}\mathscr{B}(G)}(G/G, G/H)$ has a basis with two elements
\[
\vcenter{\hbox{
\xymatrix@R=10pt@C=10pt{
& G/H \ar[dl] \ar[dr]^{\id} & \\
G/G && G/H 
}
}}
\quad \textrm{ and } \quad
\vcenter{\hbox{
\xymatrix@R=10pt@C=10pt{
& G/1 \ar[dl] \ar[dr]^{\mathrm{proj}} \\
G/G && G/H
}
}}
\]
while $\Hom_{\F_{\!2}\mathscr{B}(G)}(G/G, G/1)$ has one basis element only:
\[
\xymatrix@R=10pt@C=10pt{
& G/1 \ar[dl] \ar[dr]^{\id} \\
G/G && G/1.
}
\]
As the latter two factor through the first one, we set $f = [G/G \leftarrow G/H \overset{\id}\rightarrow G/H]$ and will prove that $fg=0$ and $fs=0$ (in the notation of \eqref{equation:vanishing-comp}). Then $g,s\in\soc \mu_{\F_{\!2}}(C_4)$, hence in view of the above discussion $\mu_{\F_{\!2}}(C_4)$ is observably not self-injective. In particular, the integral Mackey algebra $\mu_\Z(C_4)$ is not Gorenstein.

To this end, the composites $fg$ and $fs$ can be computed directly in $R\mathscr{B}(G)$ for any ring of coefficients~$R$, and the results are
\begin{align*}
fg = 2 \cdot [G/G \longleftarrow G/1 \xrightarrow{\mathrm{proj}} G/H] 
\quad \textrm{ and } \quad
fs = 2 \cdot [G/G \longleftarrow G/H \overset{\id}\la G/H] \,.
\end{align*}
In particular, they vanish in characteristic two.
\end{example}

\begin{example}[Cont.\ Example~\ref{example:B_ring_Klein}] 
\label{example:MackeyKlein_C4}
A similar inspection of Hom spaces for $G=C_2\times C_2$ reveals that, in order to prove that $g,s$ from \eqref{equation:vanishing-comp} are in the socle of $\mu_{\F_{\!2}}(G)$, it suffices to show that
$f_ig = 0$ and $f_is=0$ in $\F_{\!2}\mathscr{B}(G)$ ($i\in\{1,2,3\}$) for
\begin{equation*}
f_1 = [G/G \longleftarrow G/H \overset{\id}\la G/H]\,, \quad
f_2 = [G/G \longleftarrow G/K \overset{\id}\la G/K]\,, 
\end{equation*}
\begin{equation*}
f_3 = [G/G \longleftarrow G/L \overset{\id}\la G/L]\,,
\end{equation*}
where $H,K,L<G$ are the three subgroups of order two as in Example~\ref{example:B_ring_Klein}. By symmetry, it suffices to compute $f_1g$ and $f_1s$ only, and we get
\begin{align*}
f_1g &= 2 \cdot [G/G \longleftarrow G/1 \xrightarrow{\mathrm{proj}} G/H], \\
f_1s &= 2 \cdot [G/G \longleftarrow G/H \overset{\id}\la G/H] + 2 \cdot [G/G \longleftarrow G/1 \xrightarrow{\mathrm{proj}} G/H]
\end{align*}
(in particular we use that $G/H \times G/H \cong G/H \amalg G/H$ and $G/H \times G/K \cong G/1 \cong G/H \times G/L$ as $G$-sets). Hence again the composites vanish in characteristic two.

As before, we conclude that $\mu_{\F_{\!2}}(C_2\times C_2)$ is observably not self-injective and thus $\mu_\Z(C_2\times C_2)$ is not Gorenstein.
\end{example}

\begin{ack}
The authors would like to thank Serge Bouc, John Greenlees, Radu Stancu and Peter Symonds for useful discussions.
\end{ack}


\bibliographystyle{alpha}%
\bibliography{articles}

\end{document}